\documentclass[10pt, reqno]{amsart}
\usepackage{eurosym}
\usepackage{amsmath, amsthm, amscd, amsfonts, amssymb, graphicx, color}
\usepackage[bookmarksnumbered, colorlinks, plainpages]{hyperref}
\usepackage{graphicx, epstopdf}

\newtheorem{theorem}{Theorem}[section]
\newtheorem{lemma}[theorem]{Lemma}

\theoremstyle{definition}

\theoremstyle{remark}

\numberwithin{equation}{section}

\begin{document}
\title[Third Hankel determinant]{The Third Hankel determinant for inverse
coefficients of bounded turning functions }
\author[M. Raza, N. Tuneski]{Mohsan Raza$^{1,\ast }$, Nikola Tuneski$^{2}$}
\address{$^{1}$Department of Mathematics, Government College University
Faisalabad, Pakistan}
\email{mohsan976@yahoo.com}
\address{Faculty of Mechanical Engineering, Ss. Cyril and Methodius
University in Skopje Karpo\v{s} II b.b., 1000 Skopje, Republic of North
Macedonia}
\email{nikola.tuneski@mf.edu.mk}
\keywords{Univalent functions, bounded turning functions, Hankel
determinant, inverse coefficients}
\date{Received June 19, 2025.\\
\indent$^{\ast }$ Corresponding author\\
2020\textit{\ Mathematics Subject Classification. }30C45, 30C50.}

\begin{abstract}
In this paper, we obtain sharp bounds for the third Hankel determinants of
the coefficients of the inverse of bounded turning functions. Thus answering
a negatively to a conjecture recently posed regarding these functions.
Additionally, we offer a positive response for the Hankel determinant
related to a class of bounded turning functions.
\end{abstract}

\maketitle

\section{Introduction}

Let $\mathcal{A}$ denote the class of normalized analytic functions $f$ in
the open unit disc $\mathbb{D}:=\left\{ z:\left\vert z\right\vert <1,\ z\in 
\mathbb{C}
\right\} $ with Taylor expansion 
\begin{equation}
f(z)=z+\overset{\infty }{\sum_{n=2}}a_{n}z^{n},\ \ \ \ z\in \mathbb{D}
\label{1.1}
\end{equation}%
and let $\mathcal{S}$ denote a subclass of functions in $\mathcal{A}$ which
are univalent in $\mathbb{D}$.\ The class $\mathcal{R}$ of bounded turning
functions for $f\in \mathcal{A}$ is analytically defined as%
\begin{equation}
\mathcal{R}=\left\{ f\in \mathcal{A}:{Re}f^{\prime }\left( z\right) >0,\ \ \
z\in \mathbb{D}\right\} \text{.}  \label{bounded turning}
\end{equation}%
The condition ${Re}f^{\prime }\left( z\right) >0$ implies that $\left\vert
argf^{\prime }(z)\right\vert <\pi /2$, and $argf^{\prime }(z)$ is the angle
of rotation of the image of a line segment starting from $z$ under the
mapping $f$, see \cite[Vol 1, p. 101]{good}. Therefore, the functions in
class $\mathcal{R}$ are known as functions of bounded turnings. This class
was first studied by Noshiro \cite{Noshiro}. However, Zmorovi\v{c} \cite{zam}
explored the properties of these functions completely. It was shown in \cite%
{gal} that a function $f\in \mathcal{R}$ and of the form $\left( \ref{1.1}%
\right) $ satisfies the inequality $\left\vert a_{n}\right\vert \leq 2/n$
and equality holds for the function $f_{0}(z)=-z-2e^{-i\theta }\ln
(1-e^{-i\theta }z).$ Since the function $f_{0}\notin \mathcal{S}^{\ast }$,
the class $\mathcal{R}$ is not a subclass of the class of starlike
functions. Similarly, the well-known Koebe function $k(z)=z/(1-e^{-i\theta
}z)^{2}$ does not belongs to the class $\mathcal{R}$. Hence the class of
starlike functions is not a subclass of $\mathcal{R}$. It was told by Prof.
V. A. Zmorovi\v{c} orally during a seminar at the Department of Higher
Mathematics of the Kiev Polytechnic Institute that examples of functions of
class $\mathcal{R}$ can be constructed showing that $\mathcal{S}^{\ast }$
and $\mathcal{R}$ have common elements, but neither is contained in the
other. A one to one correspondence between these classes was given in \cite%
{gal}: If $g(z)\in \mathcal{S}^{\ast }$, then%
\begin{equation*}
f\left( z\right) =z+z\ln \left( \frac{g(z)}{z}\right) -\underset{0}{\overset{%
z}{\int }}\ln \left( \frac{g(s)}{s}\right) ds\in \mathcal{R}
\end{equation*}%
and conversely, if $f\in \mathcal{R}$, then%
\begin{equation*}
g\left( z\right) =z\exp \left( \underset{0}{\overset{z}{\int }}\frac{g(s)-s}{%
s}ds\right) \in \mathcal{S}^{\ast }\text{.}
\end{equation*}%
In 1977, Chichra \cite{Chichra} defined the class consisting of function $%
f\in \mathcal{R}_{1}$ which satisfy the condition 
\begin{equation*}
\mathcal{R}_{1}=\left\{ f\in \mathcal{A}:{Re}\left( f^{\prime }\left(
z\right) +zf^{\prime \prime }(z)\right) >0,\ \ \ z\in \mathbb{D}\right\} 
\text{.}
\end{equation*}%
which is equivalent to ${Re}\left( \left( zf^{\prime }\left( z\right)
\right) ^{\prime }\right) >0,\ z\in \mathbb{D}$. He proved that if $f\in 
\mathcal{R}_{1}$, then $Ref^{\prime }(z)>0$, and hence $f$ is univalent in $%
\mathbb{D}$. Later, Singh and Singh \cite{sigh} showed that if $f\in 
\mathcal{R}_{1}$ then $f$ is starlike in $\mathbb{D}$.

For any univalent function $f$, the inverse function $f^{-1}$ can be
expanded using the Taylor series as 
\begin{equation}
f^{-1}\left( w\right) =z+\overset{\infty }{\sum_{n=2}}A_{n}w^{n}\text{.}
\label{inverse}
\end{equation}%
This expansion is valid at least in the disc $\left\vert w\right\vert \leq
1/4$.

The $q$th Hankel determinant for analytic functions $f\in \mathcal{A}$ is
given by%
\begin{equation*}
H_{q}\left( n\right) (f):=\left\vert 
\begin{array}{llll}
a_{n} & a_{n+1} & \ldots & a_{n+q-1} \\ 
a_{n+1} & a_{n+2} & \ldots & a_{n+q} \\ 
\vdots & \vdots & \ldots & \vdots \\ 
a_{n+q-1} & a_{n+q} & \ldots & a_{n+2q-2}%
\end{array}%
\right\vert \text{,}
\end{equation*}%
where $n\geq 1$ and $q\geq 1$, see \cite{pom}. The Hankel determinant%
\begin{equation}
\left\vert H_{3}\left( 1\right) (f)\right\vert =\left\vert
2a_{2}a_{3}a_{4}-a_{3}^{3}-a_{4}^{2}+a_{3}a_{5}-a_{2}^{2}a_{5}\right\vert
\label{h31}
\end{equation}%
is known as third Hankel determinant. The sharp bounds on $\left\vert
H_{3}\left( 1\right) (f)\right\vert $ for well-known classes such as the
class $\mathcal{S}^{\ast }$ of starlike functions, the class $\mathcal{C}$
of convex functions and the class $\mathcal{R}$ of bounded turning functions
have been found in \cite{Kow,Kow2,Kow1}. The sharp bound of $\left\vert
H_{3}\left( 1\right) (f)\right\vert $ for the inverse coefficients of
starlike functions is an open problem. The sharp bound of $\left\vert
H_{3}\left( 1\right) (f)\right\vert $ for the inverse coefficients of convex
functions is obtained in \cite{Raza}. A non sharp bound for the class $%
\mathcal{R}$ is studied by Gelova and Tuneski \cite{tun}. \bigskip

In \cite{tun}, the authors have conjectured that $|H_{3}(1)(f^{-1})|\leq 1/4$
when $f\in \mathcal{R}$, and $|H_{3}(1)(f^{-1})|\leq 1/64$ for $f\in 
\mathcal{R}_{1}$. In this paper, we found sharp bounds of these results.

Let $\mathcal{P}$ denote the class of analytic functions $p$ defined for $%
z\in \mathbb{D}$ given by 
\begin{equation}
p(z)=1+\sum_{n=1}^{\infty }c_{n}z^{n}  \label{1.2}
\end{equation}%
with positive real part in $\mathbb{D}$.\ In order to prove our theorems, we
will use the following result concerning the functions in the class $%
\mathcal{P}$.

\begin{lemma}
\label{A}\cite[Lemma 2.1]{cho19} Let $\overline{\mathbb{D}}:=\{z\in \mathbb{C%
}:|z|\leq 1\}$ be the closed unit disk, and $p\in \mathcal{P}$ be given by $%
\left( \ref{1.2}\right) $. Then, 
\begin{align*}
& c_{1}=2t_{1}\text{,} \\
& c_{2}=2t_{1}^{2}+2\left( 1-|t_{1}|^{2}\right) t_{2}\text{,} \\
& c_{3}=2t_{1}^{3}+4\left( 1-|t_{1}|^{2}\right) t_{1}t_{2}-2\left(
1-|t_{1}|^{2}\right) \overline{t_{1}}t_{2}^{2}+2\left( 1-|t_{1}|^{2}\right)
\left( 1-|t_{2}|^{2}\right) t_{3}\text{,} \\
& \text{and} \\
& c_{4}=2t_{1}^{4}+2(1-|t_{1}|^{2})\left( 3t_{1}^{2}+\overline{t_{1}}%
^{2}t_{2}^{2}-3|t_{1}|^{2}t_{2}+t_{2}\right) t_{2} \\
& \quad +2(1-|t_{1}|^{2})(1-|t_{2}|^{2})\left( 2t_{1}-2\overline{t_{1}}t_{2}-%
\overline{t_{2}}t_{3}\right) t_{3} \\
& \quad +2(1-|t_{1}|^{2})(1-|t_{2}|^{2})(1-|t_{3}|^{2})t_{4}\text{,}
\end{align*}%
for some $t_{k}\in \overline{\mathbb{D}}$ for $k\in \{1,2,3,4\}$.
\end{lemma}

\section{Main Results}

We now prove the main theorems of this paper.

\begin{theorem}
If $f\in \mathcal{S}^{\ast }$ has the form $\left( \ref{1.1}\right) $. Then 
\begin{equation}
|H_{3}(1)\left( f^{-1}\right) |\leq \frac{44}{135}\text{.}  \label{th1}
\end{equation}%
This result is sharp for the function $f_{0}(z)=-z+2{arctanh}(z)$.
\end{theorem}

\begin{proof}
Let $f\in \mathcal{R}$. Then by \eqref{bounded turning}, we have 
\begin{equation}
f^{\prime }(z)=p(z)\text{,}  \label{fp1}
\end{equation}%
where $p\in \mathcal{P}$. Now using $\left( \ref{1.1}\right) $ and $\left( %
\ref{1.2}\right) $ and comparing the coefficients, we can write 
\begin{equation}
a_{2}=\frac{c_{1}}{2},\ a_{3}=\frac{1}{3}c_{2},\ a_{4}=\frac{1}{4}c_{3},\
a_{5}=\frac{1}{5}c_{4}\text{.}  \label{co4}
\end{equation}%
From $\left( \ref{inverse}\right) $, we have 
\begin{align}
A_{2}& =-a_{2},\ A_{3}=2a_{2}^{2}-a_{3},\ A_{4}=5a_{2}a_{3}-5a_{2}^{3}-a_{4}%
\text{,}  \label{in3} \\
A_{5}& =14a_{2}^{4}-21a_{3}a_{2}^{2}+6a_{2}a_{4}+3a_{3}^{2}-a_{5}\text{.}
\label{in4}
\end{align}%
Putting $\left( \ref{co4}\right) $ in $\left( \ref{in3}\right) $-$\left( \ref%
{in4}\right) $, we obtain%
\begin{align}
A_{2}& =-\frac{c_{1}}{2},\ A_{3}=\frac{1}{2}c_{1}^{2}-\frac{1}{3}c_{2},\
A_{4}=\frac{-5}{8}c_{1}^{3}+\frac{5}{6}c_{1}c_{2}-\frac{1}{4}c_{3}\text{,}
\label{ico3} \\
A_{5}& =\frac{7}{8}c_{1}^{4}-\frac{7}{4}c_{1}^{2}c_{2}+\frac{3}{4}c_{1}c_{3}+%
\frac{1}{3}c_{2}^{2}-\frac{1}{5}c_{4}\text{.}  \label{ico4}
\end{align}%
Using $\left( \ref{ico3}\right) $-$\left( \ref{ico4}\right) $ in $\left( \ref%
{h31}\right) $, we get 
\begin{equation}
H_{3}(1)\left( f^{-1}\right) =\frac{1}{8640}\left[ 
\begin{array}{c}
-540c_{1}^{4}c_{2}-432c_{1}^{2}c_{4}+720c_{1}^{2}c_{2}^{2}+576c_{2}c_{4}-540c_{3}^{2}
\\ 
+720c_{1}c_{2}c_{3}+135c_{1}^{6}-640c_{2}^{3}%
\end{array}%
\right] \text{.}  \label{iH2.1}
\end{equation}%
From Lemma \ref{A} and upon some simplification, we are able to obtain 
\begin{eqnarray*}
8640H_{3}(1)\left( f^{-1}\right) &=&208{t_{{1}}}^{6}{+}16t_{{2}}\left(
1-\left\vert t{_{{1}}}\right\vert ^{2}\right) [-132{t_{{1}}}^{4}{+6}t{_{{1}%
}^{2}}\left( 50-41\left\vert t{_{{1}}}\right\vert ^{2}\right) t_{{2}} \\
&&-4\left( 35\left\vert t{_{{1}}}\right\vert ^{4}{-}61\left\vert t{_{{1}}}%
\right\vert ^{2}+44\right) t{_{{2}}^{2}}+9\overline{t{_{{1}}}}^{2}\left(
1-\left\vert t{_{{1}}}\right\vert ^{2}\right) t{_{{2}}^{3}]} \\
&&-288\left( 1-\left\vert t{_{{1}}}\right\vert ^{2}\right) \left(
1-\left\vert t_{{2}}\right\vert ^{2}\right) [\left( 1-\left\vert t{_{{1}}}%
\right\vert ^{2}\right) \overline{t{_{{1}}}}t{_{{2}}^{2}}-2\left(
3+\left\vert t{_{{1}}}\right\vert ^{2}\right) t{_{{1}}}t_{{2}}+3t{_{{1}}^{3}]%
}t_{{3}} \\
&&-144\left( 1-\left\vert t{_{{1}}}\right\vert ^{2}\right) \left(
1-\left\vert t_{{2}}\right\vert ^{2}\right) [\left( 1-t{_{{1}}^{2}}\right)
\left( \left\vert t_{{2}}\right\vert ^{2}+15\right) -8t{_{{1}}^{2}}\overline{%
t_{{2}}}]t_{{3}}^{2} \\
&&+1152\left( 1-\left\vert t{_{{1}}}\right\vert ^{2}\right) \left(
1-\left\vert t_{{2}}\right\vert ^{2}\right) \left( 1-\left\vert t_{{3}%
}\right\vert ^{2}\right) [2\left( 1-\left\vert t{_{{1}}}\right\vert
^{2}\right) t_{{2}}-t{_{{1}}^{2}]}t_{{4}}
\end{eqnarray*}%
for some $t_{k}\in \overline{\mathbb{D}},\ k=1,2,3,4.$ Now applying triangle
inequality with $\left\vert t_{4}\right\vert \leq 1$, it follows that 
\begin{eqnarray*}
8640\left\vert H_{3}(1)\left( f^{-1}\right) \right\vert &\leq &208\left\vert
t{_{{1}}}\right\vert ^{6}{+}16\left\vert t_{{2}}\right\vert \left(
1-\left\vert t{_{{1}}}\right\vert ^{2}\right) [132\left\vert t{_{{1}}}%
\right\vert ^{4}{+6}\left\vert t{_{{1}}}\right\vert ^{2}\left(
50-41\left\vert t{_{{1}}}\right\vert ^{2}\right) \left\vert t_{{2}%
}\right\vert \\
&&+4\left( 35\left\vert t{_{{1}}}\right\vert ^{4}{-61}\left\vert t{_{{1}}}%
\right\vert ^{2}+44\right) \left\vert t{_{{2}}}\right\vert ^{2}+9\left\vert t%
{_{{1}}}\right\vert ^{2}\left( 1-\left\vert t{_{{1}}}\right\vert ^{2}\right)
\left\vert t{_{{2}}}\right\vert ^{3}{]} \\
&&+288\left( 1-\left\vert t{_{{1}}}\right\vert ^{2}\right) \left(
1-\left\vert t_{{2}}\right\vert ^{2}\right) [\left( 1-\left\vert t{_{{1}}}%
\right\vert ^{2}\right) \left\vert t{_{{1}}}\right\vert \left\vert t{_{{2}}}%
\right\vert ^{2} \\
&&+2\left( 3+\left\vert t{_{{1}}}\right\vert ^{2}\right) \left\vert t{_{{1}}}%
\right\vert \left\vert t_{{2}}\right\vert +3\left\vert t{_{{1}}}\right\vert
^{3}{]}\left\vert t_{{3}}\right\vert \\
&&+144\left( 1-\left\vert t{_{{1}}}\right\vert ^{2}\right) \left(
1-\left\vert t_{{2}}\right\vert ^{2}\right) \left\vert \left( 1-t{_{{1}}^{2}}%
\right) \left( 15+\left\vert t_{{2}}\right\vert ^{2}\right) -8t{_{{1}}^{2}}%
\overline{t_{{2}}}\right\vert \left\vert t_{{3}}\right\vert ^{2} \\
&&+1152\left( 1-\left\vert t{_{{1}}}\right\vert ^{2}\right) \left(
1-\left\vert t_{{2}}\right\vert ^{2}\right) \left( 1-\left\vert t_{{3}%
}\right\vert ^{2}\right) \left\vert 2\left( 1-\left\vert t{_{{1}}}%
\right\vert ^{2}\right) t_{{2}}-t{_{{1}}^{2}}\right\vert \text{.}
\end{eqnarray*}%
Therefore%
\begin{eqnarray}
8640\left\vert H_{3}(1)\left( f^{-1}\right) \right\vert &\leq &208\left\vert
t{_{{1}}}\right\vert ^{6}{+}16\left\vert t_{{2}}\right\vert \left(
1-\left\vert t{_{{1}}}\right\vert ^{2}\right) [132\left\vert t{_{{1}}}%
\right\vert ^{4}{+6}\left\vert t{_{{1}}}\right\vert ^{2}\left(
50-41\left\vert t{_{{1}}}\right\vert ^{2}\right) \left\vert t_{{2}%
}\right\vert  \notag \\
&&+4\left( 35\left\vert t{_{{1}}}\right\vert ^{4}{-61}\left\vert t{_{{1}}}%
\right\vert ^{2}+44\right) \left\vert t{_{{2}}}\right\vert ^{2}+9\left\vert t%
{_{{1}}}\right\vert ^{2}\left( 1-\left\vert t{_{{1}}}\right\vert ^{2}\right)
\left\vert t{_{{2}}}\right\vert ^{3}{]}  \notag \\
&&+288\left( 1-\left\vert t{_{{1}}}\right\vert ^{2}\right) \left(
1-\left\vert t_{{2}}\right\vert ^{2}\right) [\left( 1-\left\vert t{_{{1}}}%
\right\vert ^{2}\right) \left\vert t{_{{1}}}\right\vert \left\vert t{_{{2}}}%
\right\vert ^{2}  \notag \\
&&+2\left( 3+\left\vert t{_{{1}}}\right\vert ^{2}\right) \left\vert t{_{{1}}}%
\right\vert \left\vert t_{{2}}\right\vert +3\left\vert t{_{{1}}}\right\vert
^{3}{]}\left\vert t_{{3}}\right\vert  \notag \\
&&+144\left( 1-\left\vert t{_{{1}}}\right\vert ^{2}\right) \left(
1-\left\vert t_{{2}}\right\vert ^{2}\right) \left\vert \left( 1-t{_{{1}}^{2}}%
\right) \left( 15+\left\vert t_{{2}}\right\vert ^{2}\right) -8t{_{{1}}^{2}}%
\overline{t_{{2}}}\right\vert \left\vert t_{{3}}\right\vert ^{2}  \notag \\
&&-8\left\vert 2\left( 1-\left\vert t{_{{1}}}\right\vert ^{2}\right) t_{{2}%
}-t{_{{1}}^{2}}\right\vert ]\left\vert t_{{3}}\right\vert ^{2}  \notag \\
&&+1152\left( 1-\left\vert t{_{{1}}}\right\vert ^{2}\right) \left(
1-\left\vert t_{{2}}\right\vert ^{2}\right) \left\vert 2\left( 1-\left\vert t%
{_{{1}}}\right\vert ^{2}\right) t_{{2}}-t{_{{1}}^{2}}\right\vert \text{.}
\label{m}
\end{eqnarray}%
\textbf{\ A.} Firstly, assume that 
\begin{equation*}
\left\vert \left( 1-\left\vert t{_{{1}}}\right\vert ^{2}\right) \left(
15+\left\vert t_{{2}}\right\vert ^{2}\right) -8t{_{{1}}^{2}}\overline{t_{{2}}%
}\right\vert -8\left\vert 2\left( 1-\left\vert t{_{{1}}}\right\vert
^{2}\right) t_{{2}}-t{_{{1}}^{2}}\right\vert \geq 0.
\end{equation*}%
Now using the fact that $\left\vert t_{{3}}\right\vert \leq 1$, we can write%
\begin{eqnarray*}
8640\left\vert H_{3}(1)\left( f^{-1}\right) \right\vert &\leq &208\left\vert
t{_{{1}}}\right\vert ^{6}{+}16\left\vert t_{{2}}\right\vert \left(
1-\left\vert t{_{{1}}}\right\vert ^{2}\right) [132\left\vert t{_{{1}}}%
\right\vert ^{4}{+6}\left\vert t{_{{1}}}\right\vert ^{2}\left(
50-41\left\vert t{_{{1}}}\right\vert ^{2}\right) \left\vert t_{{2}%
}\right\vert \\
&&+4\left( 35\left\vert t{_{{1}}}\right\vert ^{4}{-61}\left\vert t{_{{1}}}%
\right\vert ^{2}+44\right) \left\vert t{_{{2}}}\right\vert ^{2}+9\left\vert t%
{_{{1}}}\right\vert ^{2}\left( 1-\left\vert t{_{{1}}}\right\vert ^{2}\right)
\left\vert t{_{{2}}}\right\vert ^{3}{]} \\
&&+288\left( 1-\left\vert t{_{{1}}}\right\vert ^{2}\right) \left(
1-\left\vert t_{{2}}\right\vert ^{2}\right) [\left( 1-\left\vert t{_{{1}}}%
\right\vert ^{2}\right) \left\vert t{_{{1}}}\right\vert \left\vert t{_{{2}}}%
\right\vert ^{2} \\
&&+2\left( 3+\left\vert t{_{{1}}}\right\vert ^{2}\right) \left\vert t{_{{1}}}%
\right\vert \left\vert t_{{2}}\right\vert +3\left\vert t{_{{1}}}\right\vert
^{3}{]} \\
&&+144\left( 1-\left\vert t{_{{1}}}\right\vert ^{2}\right) \left(
1-\left\vert t_{{2}}\right\vert ^{2}\right) \left\vert \left( 1-\left\vert t{%
_{{1}}}\right\vert ^{2}\right) \left( 15+\left\vert t_{{2}}\right\vert
^{2}\right) -8t{_{{1}}^{2}}\overline{t_{{2}}}\right\vert .
\end{eqnarray*}%
This implies that 
\begin{equation*}
\left\vert H_{3}(1)\left( f^{-1}\right) \right\vert \leq \frac{1}{8640}%
g\left( \left\vert t{_{{1}}}\right\vert {,}\left\vert t{_{{2}}}\right\vert
\right) \text{,}
\end{equation*}%
where $g:\mathbb{R}^{2}\rightarrow \mathbb{R}$ is defined by 
\begin{eqnarray*}
g\left( s,u\right) &\leq &208s^{6}{+}16u\left( 1-s^{2}\right) [132s^{4}{+6}%
s^{2}\left( 50-41s^{2}\right) u \\
&&+4\left( 35s^{4}{-}61s^{2}+44\right) u^{2}+9s^{2}\left( 1-s^{2}\right)
u^{3}{]} \\
&&+288s\left( 1-s^{2}\right) \left( 1-u^{2}\right) [\left( 1-s^{2}\right)
u^{2}+2\left( 3+s^{2}\right) u+3s^{2}{]} \\
&&+144\left( 1-s^{2}\right) \left( 1-u^{2}\right) [\left( 1-s^{2}\right)
\left( 15+u^{2}\right) +8s{^{2}}u]\text{.}
\end{eqnarray*}%
We prove that $\max g\left( s,u\right) =2816$ for $\left( s,u\right) \in %
\left[ 0,1\right] \times \left[ 0,1\right] $.

\textbf{I}. On the vertices of $\left[ 0,1\right] \times \left[ 0,1\right] $%
, we have the following cases%
\begin{equation*}
g\left( {0,}0\right) =2160\text{,}\ g(0,1)=2816\text{,}\ g(1,0)=208\text{,}\
g(1,1)=208\text{.}
\end{equation*}%
\textbf{II}. On the sides of $\left[ 0,1\right] \times \left[ 0,1\right] $,
we have

(i) $g(0,u)=2816u^{3}+144(1-u^{2})(15+u^{2})\leq g(0,1)=2816,$ $u\in \left(
0,1\right) $.

(ii) $g(1,u)=208,$ $u\in \left( 0,1\right) $.

(iii) $g(s,0)=208s^{6}+432(1-s^{2})(2s^{3}-5s^{2}+5)\leq g(0,0)=2160,~s\in
\left( 0,1\right) $.

(iv) $g(s,1)=-16(2s^{2}+11)(2s^{4}+13s^{2}-16)\leq g(0,1)=2816,~s\in \left(
0,1\right) .$

\textbf{III. }Now we consider the case when $\left( s,u\right) \in \left(
0,1\right) \times \left( 0,1\right) $.$\ $Therefore 
\begin{eqnarray*}
\frac{\partial g}{\partial s} &=&288(1-s^{2})(-3s^{3}+5s^{2}+3s-1)u^{4} \\
&&-192(70s^{5}-15s^{4}-152s^{3}-18s^{2}+82s+9)u^{3} \\
&&+96(s-1)(246s^{4}+306s^{3}-142s^{2}-187s-3)u^{2} \\
&&-192(66s^{5}+15s^{4}-20s^{3}+18s^{2}-12s-9)u \\
&&+96s(13s^{4}-45s^{3}+90s^{2}+27s-90)\text{,}
\end{eqnarray*}%
and 
\begin{eqnarray*}
\frac{\partial g}{\partial u} &=&576(s-1)^{2}(s+1)^{2}(s^{2}-2s-1)u^{3} \\
&&+192(1-s^{2})(35s^{4}-9s^{3}-79s^{2}-27s+44)u^{2} \\
&&+192(s-1)^{2}(s+1)(41s^{3}+53s^{2}-18s-21)u \\
&&+192s(1-s^{2})(11s^{3}+3s^{2}+6s+9)\text{.}
\end{eqnarray*}%
The system of equations $\frac{\partial g}{\partial s}=0$ and $\frac{%
\partial g}{\partial u}=0$ have the following numerical solutions 
\begin{eqnarray*}
&&\left\{ 
\begin{array}{c}
s_{1}\approx 1.00000, \\ 
u_{1}\approx 1.57568,%
\end{array}%
\right. \ \ \ \ \ \ \ \left\{ 
\begin{array}{c}
s_{2}\approx 0.12412, \\ 
u_{2}\approx 0.44526,%
\end{array}%
\right. \ \ \ \ \ \ \ \left\{ 
\begin{array}{c}
s_{3}\approx 3.288940, \\ 
u_{3}\approx 28.52209,%
\end{array}%
\right. \ \ \ \left\{ 
\begin{array}{c}
s_{4}\approx -1.00000, \\ 
u_{4}\approx 0.560770,%
\end{array}%
\right. \\
&&\left\{ 
\begin{array}{c}
s_{5}\approx -1.37992, \\ 
u_{5}\approx 8.148040,%
\end{array}%
\right. \ \ \ \ \ \left\{ 
\begin{array}{c}
s_{6}\approx 1.000000, \\ 
u_{6}\approx -0.04314,%
\end{array}%
\right. \ \ \ \ \ \left\{ 
\begin{array}{c}
s_{7}\approx 0.827620, \\ 
u_{7}\approx -0.87272,%
\end{array}%
\right. \ \ \left\{ 
\begin{array}{c}
s_{8}\approx 0.785960, \\ 
u_{8}\approx -0.93064,%
\end{array}%
\right. \\
&&\left\{ 
\begin{array}{c}
s_{9}\approx 1.000000, \\ 
u_{9}\approx -1.53255,%
\end{array}%
\right. \ \ \ \ \ \left\{ 
\begin{array}{c}
s_{10}\approx 0, \\ 
u_{10}\approx 0.%
\end{array}%
\right. \ \ \ \ 
\end{eqnarray*}%
The only point in $\left( 0,1\right) \times \left( 0,1\right) $ is $\left(
s_{2},u_{2}\right) $ and we see that $g(s_{2},u_{2})=2043.962299.$ Hence $g$
has no critical point in $\left( 0,1\right) \times \left( 0,1\right) .$%
\vspace{1.5mm}

\textbf{B.} Now we consider the case when 
\begin{equation*}
\left\vert \left( 1-\left\vert t{_{{1}}}\right\vert ^{2}\right) \left(
15+\left\vert t_{{2}}\right\vert ^{2}\right) -8t{_{{1}}^{2}}\overline{t_{{2}}%
}\right\vert -8\left\vert 2\left( 1-\left\vert t{_{{1}}}\right\vert
^{2}\right) t_{{2}}-t{_{{1}}^{2}}\right\vert <0.
\end{equation*}%
Then using the fact that $\left\vert t_{{3}}\right\vert \leq 1$ we obtain
from $\left( \ref{m}\right) $ that 
\begin{equation*}
\left\vert H_{3}(1)\left( f^{-1}\right) \right\vert \leq \frac{1}{8640}%
g_{1}\left( \left\vert t{_{{1}}}\right\vert {,}\left\vert t{_{{2}}}%
\right\vert \right) ,
\end{equation*}%
where 
\begin{eqnarray*}
g_{1}\left( s,u\right) &\leq &208s^{6}{+}16u\left( 1-s^{2}\right) [132s^{4}{%
+6}s^{2}\left( 50-41s^{2}\right) u \\
&&+4\left( 35s^{4}{-}61s^{2}+44\right) u^{2}+9s^{2}\left( 1-s^{2}\right)
u^{3}{]} \\
&&+288s\left( 1-s^{2}\right) \left( 1-u^{2}\right) [\left( 1-s^{2}\right)
u^{2}+2\left( 3+s^{2}\right) u+3s^{2}{]} \\
&&+1152\left( 1-s^{2}\right) \left( 1-u^{2}\right) [2\left( 1-s^{2}\right)
u+s{^{2}]}.
\end{eqnarray*}%
We now show that $\max g_{1}\left( s,u\right) =2816$ when $\left( s,u\right)
\in \left[ 0,1\right] \times \left[ 0,1\right] .$\vspace{1.2mm}

\textbf{I}. On the vertices of $\left[ 0,1\right] \times \left[ 0,1\right] $%
, we have 
\begin{equation*}
g_{1}\left( {0,}0\right) =0,\ \ \ g_{1}(0,1)=2816,\ \ \ g_{1}(1,0)=208,\ \ \
g_{1}(1,1)=208\text{.}
\end{equation*}%
\textbf{II.} On the sides of $\left[ 0,1\right] \times \left[ 0,1\right] $,
we have the following cases

(i) $g_{1}(0,u)=2816u^{3}+2304u(1-u^{2})\leq g_{1}(0,1)=2816,$ $u\in \left(
0,1\right) .$

(ii) $g_{1}(1,u)=208,$ $u\in \left( 0,1\right) .$

(iii) $g_{1}(s,0)=208s^{6}+288s^{2}(1-s^{2})(3s+4)\lesssim
g_{1}(0.7734436204,0)\approx 482.0335053,~s\in \left( 0,1\right) $.

(iv) $g_{1}(s,1)=-16(2x^{2}+11)(2x^{4}+13x^{2}-16)\leq g_{1}(0,1)=2816,~s\in
\left( 0,1\right) $.

\textbf{III.} Now we consider the case for $\left( s,u\right) \in \left(
0,1\right) \times \left( 0,1\right) .$ We find all solutions of system of
nonlinear equations 
\begin{eqnarray*}
\frac{\partial g_{1}}{\partial s} &=&288(s^{2}-1)(3s^{3}-5s^{2}-s+1)u^{4} \\
&&-192(70s^{5}-15s^{4}-80s^{3}-18s^{2}+22s+9)u^{3} \\
&&+96(246s^{5}+60s^{4}-316s^{3}-45s^{2}+76s+3)u^{2} \\
&&-192(66s^{5}+15s^{4}-92s^{3}+18s^{2}+48s-9)u \\
&&+96s(13s^{4}-45s^{3}-48s^{2}+27s+24) \\
&=&0,
\end{eqnarray*}%
and 
\begin{eqnarray*}
\frac{\partial g_{1}}{\partial u} &=&576s(s+1)^{2}(s-1)^{2}(s-2)u^{3} \\
&&+192(1-s^{2})(35s^{4}-9s^{3}-25s^{2}-27s+8)u^{2} \\
&&-192s(1-s^{2})(41s^{3}+12s^{2}-38s-3)u \\
&&+192(1-s^{2})(11s^{4}+3s^{3}-12s^{2}+9s+12) \\
&=&0.
\end{eqnarray*}%
We have the following solutions%
\begin{eqnarray*}
&&\left\{ 
\begin{array}{c}
s_{1}\approx 2.903610, \\ 
u_{1}\approx 32.45641,%
\end{array}%
\right. \ \ \ \ \ \ \ \left\{ 
\begin{array}{c}
s_{2}\approx 1.00000, \\ 
u_{2}\approx 1.69046,%
\end{array}%
\right. \ \ \ \ \ \ \ \left\{ 
\begin{array}{c}
s_{3}\approx -1.36137, \\ 
u_{3}\approx 11.74973,%
\end{array}%
\right. \ \ \ \ \ \ \left\{ 
\begin{array}{c}
s_{4}\approx -1.00000, \\ 
u_{4}\approx 0.44619,%
\end{array}%
\right.  \\
&&\left\{ 
\begin{array}{c}
s_{5}\approx -5.49327, \\ 
u_{5}\approx 0.701860,%
\end{array}%
\right. \ \ \ \ \ \ \ \left\{ 
\begin{array}{c}
s_{6}\approx -4.52812, \\ 
u_{6}\approx 0.481570,%
\end{array}%
\right. \ \ \ \ \ \left\{ 
\begin{array}{c}
s_{7}\approx -0.50175, \\ 
u_{7}\approx -6.25137,%
\end{array}%
\right. \ \ \ \ \ \left\{ 
\begin{array}{c}
s_{8}\approx 1.000000, \\ 
u_{8}\approx -0.29888,%
\end{array}%
\right.  \\
&&\left\{ 
\begin{array}{c}
s_{9}\approx 1.000000, \\ 
u_{9}\approx -2.39158.%
\end{array}%
\right. 
\end{eqnarray*}%
This shows that there is no point of maxima in $\left( 0,1\right) \times
\left( 0,1\right) .$ All cases have been dealt with, $\left( \ref{th1}%
\right) $ holds. 

To see that $\left( \ref{th1}\right) $ is sharp, consider $f_{0}:\mathbb{D}%
\rightarrow 
\mathbb{C}
$ given by 
\begin{equation*}
f_{0}\left( z\right) =-z+2{arctanh}(z)=z+\frac{2}{3}z^{3}+\frac{2}{5}%
z^{5}+\cdots .
\end{equation*}%
It implies that $f_{0}^{\prime }\left( z\right) =\frac{1+z^{2}}{1-z^{2}}.$
Therefore $f_{0}\in \mathcal{R}$ with $a_{2}=a_{4}=0$ and $a_{3}=\dfrac{2}{3}%
,\ a_{5}=\dfrac{2}{5}$ which from $\left( \ref{in3}\right) -\left( \ref{in4}%
\right) $ gives $A_{2}=A_{4}=0$ and $A_{3}=-\dfrac{2}{3},\ A_{5}=\frac{14}{15%
}.$ Hence from $\left( \ref{h31}\right) ,$ we obtain $|H_{3}(1)\left(
f_{0}^{-1}\right) |=\frac{44}{135}.$ This completes the proof of the theorem.
\end{proof}

\begin{theorem}
Let $f\in \mathcal{R}_{1}$ has the form $\left( \ref{1.1}\right) $. Then 
\begin{equation}
|H_{3}(1)\left( f^{-1}\right) |\leq \frac{1}{64}.  \label{H3.0}
\end{equation}%
The result is sharp for the function $f_{\ast }$ given by%
\begin{equation*}
\left( zf_{\ast }^{\prime }\left( z\right) \right) ^{\prime }=\frac{1+z^{3}}{%
1-z^{3}}.
\end{equation*}
\end{theorem}

\begin{proof}
Using $\left( \ref{ico3}\right) -\left( \ref{ico4}\right) $ in $\left( \ref%
{h31}\right) $, we obtain 
\begin{eqnarray}
H_{3}(1)\left( f^{-1}\right) &=&\frac{1}{74649600}%
[-97200c_{1}^{4}c_{2}-186624c_{1}^{2}c_{4}+172800c_{1}^{2}c_{2}^{2}+331776c_{2}c_{4}-291600c_{3}^{2}
\notag \\
&&+18225c_{1}^{6}+259200c_{1}c_{2}c_{3}-204800c_{2}^{3}].  \label{H3.1}
\end{eqnarray}%
Using Lemma \ref{A}, in $\left( \ref{H3.1}\right) $, we obtain 
\begin{eqnarray}
74649600H_{3}(1)\left( f^{-1}\right) &=&-76288{t_{{1}}}^{6}{+}64t_{{2}%
}\left( 1-\left\vert t{_{{1}}}\right\vert ^{2}\right) [-1740{t_{{1}}}^{4}{+6}%
t{_{{1}}^{2}}\left( 2986-1447\left\vert t{_{{1}}}\right\vert ^{2}\right) t_{{%
2}}  \notag \\
&&-4\left( 1621\left\vert t{_{{1}}}\right\vert ^{4}-2189\left\vert t{_{{1}}}%
\right\vert ^{2}+1216\right) t{_{{2}}^{2}}+2511\overline{t{_{{1}}}}%
^{2}\left( 1-\left\vert t{_{{1}}}\right\vert ^{2}\right) t{_{{2}}^{3}]} 
\notag \\
&&-10368\left( 1-\left\vert t{_{{1}}}\right\vert ^{2}\right) \left(
1-\left\vert t_{{2}}\right\vert ^{2}\right) [31\left( 1-\left\vert t{_{{1}}}%
\right\vert ^{2}\right) \overline{t{_{{1}}}}t{_{{2}}^{2}}  \notag \\
&&-2\left( 3+13\left\vert t{_{{1}}}\right\vert ^{2}\right) t{_{{1}}}t_{{2}%
}+57t{_{{1}}^{3}]}t_{{3}}  \notag \\
&&-5184\left( 1-\left\vert t{_{{1}}}\right\vert ^{2}\right) \left(
1-\left\vert t_{{2}}\right\vert ^{2}\right) [\left( 1-t{_{{1}}^{2}}\right)
\left( 31\left\vert t_{{2}}\right\vert ^{2}+225\right) -32t{_{{1}}^{2}}%
\overline{t_{{2}}}]t_{{3}}^{2}  \notag \\
&&+165888\left( 1-\left\vert t{_{{1}}}\right\vert ^{2}\right) \left(
1-\left\vert t_{{2}}\right\vert ^{2}\right) \left( 1-\left\vert t_{{3}%
}\right\vert ^{2}\right)  \notag \\
&&\lbrack 8\left( 1-\left\vert t{_{{1}}}\right\vert ^{2}\right) t_{{2}}-t{_{{%
1}}^{2}]}t_{{4}}  \label{H3.2}
\end{eqnarray}%
for some $t_{1}\in \lbrack 0,1]$ and $t_{2},t_{3},t_{4}\in \overline{\mathbb{%
D}}$. Since $|t_{4}|\leq 1$, using the triangle inequality in $\left( \ref%
{H3.2}\right) $, we obtain 
\begin{eqnarray*}
74649600H_{3}(1)\left( f^{-1}\right) &=&76288\left\vert {t_{{1}}}\right\vert
^{6}{+}64\left\vert t_{{2}}\right\vert \left( 1-\left\vert t{_{{1}}}%
\right\vert ^{2}\right) [1740\left\vert {t_{{1}}}\right\vert ^{4}{+6}%
\left\vert t{_{{1}}}\right\vert ^{2}\left( 2986-1447\left\vert t{_{{1}}}%
\right\vert ^{2}\right) \left\vert t_{{2}}\right\vert \\
&&+4\left( 1621\left\vert t{_{{1}}}\right\vert ^{4}-2189\left\vert t{_{{1}}}%
\right\vert ^{2}+1216\right) \left\vert t{_{{2}}}\right\vert
^{2}+2511\left\vert t{_{{1}}}\right\vert ^{2}\left( 1-\left\vert t{_{{1}}}%
\right\vert ^{2}\right) \left\vert t{_{{2}}}\right\vert ^{3}{]} \\
&&+10368\left( 1-\left\vert t{_{{1}}}\right\vert ^{2}\right) \left(
1-\left\vert t_{{2}}\right\vert ^{2}\right) [31\left( 1-\left\vert t{_{{1}}}%
\right\vert ^{2}\right) \left\vert t{_{{1}}}\right\vert \left\vert t{_{{2}}}%
\right\vert ^{2} \\
&&+2\left( 3+13\left\vert t{_{{1}}}\right\vert ^{2}\right) \left\vert t{_{{1}%
}}\right\vert \left\vert t_{{2}}\right\vert +57\left\vert t{_{{1}}}%
\right\vert ^{3}{]}\left\vert t_{{3}}\right\vert \\
&&+5184\left( 1-\left\vert t{_{{1}}}\right\vert ^{2}\right) \left(
1-\left\vert t_{{2}}\right\vert ^{2}\right) \left\vert \left( 1-t{_{{1}}^{2}}%
\right) \left( 31\left\vert t_{{2}}\right\vert ^{2}+225\right) -32t{_{{1}%
}^{2}}\overline{t_{{2}}}]\right\vert \left\vert t_{{3}}\right\vert ^{2} \\
&&+165888\left( 1-\left\vert t{_{{1}}}\right\vert ^{2}\right) \left(
1-\left\vert t_{{2}}\right\vert ^{2}\right) \left( 1-\left\vert t_{{3}%
}\right\vert ^{2}\right) \left\vert 8\left( 1-\left\vert t{_{{1}}}%
\right\vert ^{2}\right) t_{{2}}-t{_{{1}}^{2}}\right\vert
\end{eqnarray*}%
Therefore%
\begin{eqnarray}
74649600H_{3}(1)\left( f^{-1}\right) &=&76288\left\vert {t_{{1}}}\right\vert
^{6}{+}64\left\vert t_{{2}}\right\vert \left( 1-\left\vert t{_{{1}}}%
\right\vert ^{2}\right) [1740\left\vert {t_{{1}}}\right\vert ^{4}{+6}%
\left\vert t{_{{1}}}\right\vert ^{2}\left( 2986-1447\left\vert t{_{{1}}}%
\right\vert ^{2}\right) \left\vert t_{{2}}\right\vert  \notag \\
&&+4\left( 1621\left\vert t{_{{1}}}\right\vert ^{4}-2189\left\vert t{_{{1}}}%
\right\vert ^{2}+1216\right) \left\vert t{_{{2}}}\right\vert
^{2}+2511\left\vert t{_{{1}}}\right\vert ^{2}\left( 1-\left\vert t{_{{1}}}%
\right\vert ^{2}\right) \left\vert t{_{{2}}}\right\vert ^{3}{]}  \notag \\
&&+10368\left( 1-\left\vert t{_{{1}}}\right\vert ^{2}\right) \left(
1-\left\vert t_{{2}}\right\vert ^{2}\right) [31\left( 1-\left\vert t{_{{1}}}%
\right\vert ^{2}\right) \left\vert t{_{{1}}}\right\vert \left\vert t{_{{2}}}%
\right\vert ^{2}  \notag \\
&&+2\left( 3+13\left\vert t{_{{1}}}\right\vert ^{2}\right) \left\vert t{_{{1}%
}}\right\vert \left\vert t_{{2}}\right\vert +57\left\vert t{_{{1}}}%
\right\vert ^{3}{]}\left\vert t_{{3}}\right\vert  \notag \\
&&+5184\left( 1-\left\vert t{_{{1}}}\right\vert ^{2}\right) \left(
1-\left\vert t_{{2}}\right\vert ^{2}\right) [\left\vert \left( 1-\left\vert t%
{_{{1}}}\right\vert ^{2}\right) \left( 31\left\vert t_{{2}}\right\vert
^{2}+225\right) -32t{_{{1}}^{2}}\overline{t_{{2}}}]\right\vert  \notag \\
&&-32\left\vert 8\left( 1-\left\vert t{_{{1}}}\right\vert ^{2}\right) t_{{2}%
}-t{_{{1}}^{2}}\right\vert ]\left\vert t_{{3}}\right\vert ^{2}  \notag \\
&&+165888\left( 1-\left\vert t{_{{1}}}\right\vert ^{2}\right) \left(
1-\left\vert t_{{2}}\right\vert ^{2}\right) \left\vert 8\left( 1-\left\vert t%
{_{{1}}}\right\vert ^{2}\right) t_{{2}}-t{_{{1}}^{2}}\right\vert .
\label{H3.3}
\end{eqnarray}%
\textbf{A}. Suppose that 
\begin{equation*}
\left\vert \left( 1-t{_{{1}}^{2}}\right) \left( 31\left\vert t_{{2}%
}\right\vert ^{2}+225\right) -32t{_{{1}}^{2}}\overline{t_{{2}}}]\right\vert
-32\left\vert 8\left( 1-\left\vert t{_{{1}}}\right\vert ^{2}\right) t_{{2}}-t%
{_{{1}}^{2}}\right\vert \geq 0.
\end{equation*}%
Then, since $|t_{3}|\leq 1$, we obtain, from $\left( \ref{H3.3}\right) $
that 
\begin{equation*}
|H_{3}(1)\left( f^{-1}\right) |\leq \frac{1}{74649600}h(\left\vert
t_{1}\right\vert ,|t_{2}|),
\end{equation*}%
where $h:\mathbb{R}^{2}\rightarrow \mathbb{R}$ is defined by 
\begin{eqnarray}
h(s,u) &=&76288s^{6}{+}64u\left( 1-s^{2}\right) [1740s^{4}{+6}s^{2}\left(
2986-1447s^{2}\right) u  \notag \\
&&+4\left( 1621s^{4}-2189s^{2}+1216\right) u^{2}+2511s^{2}\left(
1-s^{2}\right) u^{3}{]}  \notag \\
&&+10368\left( 1-s^{2}\right) \left( 1-u^{2}\right) [31\left( 1-s^{2}\right)
su^{2}+2\left( 3+13s^{2}\right) su+57s^{3}{]}  \notag \\
&&+5184\left( 1-s^{2}\right) \left( 1-u^{2}\right) [\left( 1-s^{2}\right)
\left( 31u^{2}+225\right) +32s^{2}u].  \notag
\end{eqnarray}

We show that $|h(s,u)|\leq 1166400$ for $(s,u)\in \lbrack 0,1]\times \lbrack
0,1]$.\newline
\textbf{I.} On the vertices of $[0,1]\times \lbrack 0,1]$, we have 
\begin{equation*}
h(0,0)=1166400,\ h(0,1)=311296,\ h(1,0)=76288,\ h(1,1)=76288.
\end{equation*}%
\newline
\textbf{II}. On the sides of $\left[ 0,1\right] \times \left[ 0,1\right] $,
we have the following cases

(i) $h(0,u)=311296u^{3}+518(1-u^{2})(225+31u^{2})\leq h(0,0)=1166400,$ $u\in
\left( 0,1\right) .$\vspace{1mm}

(ii) $h(1,u)=76288,$ $u\in \left( 0,1\right) .$\vspace{1mm}

(iii) $h(s,0)=76288s^{6}+15552(1-s^{2})(38s^{3}-75s^{2}+75)\leq
h(0,0)=1166400,~s\in \left( 0,1\right) .$

(iv) $h(s,1)=266304s^{6}-936960s^{4}+435648s^{2}+311296\lesssim
h(0,0.51157)\simeq 365908.5766,~s\in \left( 0,1\right) .$

\textbf{III. }Now we consider the case when $\left( s,u\right) \in \left(
0,1\right) \times \left( 0,1\right) .\ $Therefore 
\begin{eqnarray*}
\frac{\partial h}{\partial s} &=&321408(s^{2}-1)(3s^{3}-5s^{2}-3s+1)u^{4} \\
&&-768(3242s^{5}-1755s^{4}-5944s^{3}+810s^{2}+2702s+81)u^{3} \\
&&+384(s-1)(8682s^{4}+20562s^{3}-7646s^{2}-17285s-837)u^{2} \\
&&-768(+870s^{5}+1755s^{4}+284s^{3}-810s^{2}-432s-81)u \\
&&+384s(1192s^{4}-7695s^{3}+12150s^{2}+4617s-12150),
\end{eqnarray*}%
and 
\begin{eqnarray*}
\frac{\partial h}{\partial u} &=&642816(s+1)^{2}(s-1)^{2}(s^{2}-2s-1)u^{3} \\
&&+768(1-s^{2})(1621s^{4}-1053s^{3}-2837s^{2}-243s+1216)u^{2} \\
&&+768(1+s)(s-1)^{2}(1447s^{3}+3823s^{2}-1782s-2619)u \\
&&+768s(1-s^{2})(145s^{3}+351s^{2}+216s+81).
\end{eqnarray*}%
Thus the system of equations $\frac{\partial h}{\partial s}=0$ and $\frac{%
\partial h}{\partial u}=0$ have the following numerical solutions 
\begin{eqnarray*}
&&\left\{ 
\begin{array}{c}
s_{1}\approx 1.41212, \\ 
u_{1}\approx 1.41974,%
\end{array}%
\right. \ \ \ \ \ \ \ \left\{ 
\begin{array}{c}
s_{2}\approx 1.00000, \\ 
u_{2}\approx 1.58833,%
\end{array}%
\right. \ \ \ \ \ \ \ \left\{ 
\begin{array}{c}
s_{3}\approx 5.05013, \\ 
u_{3}\approx 0.14780,%
\end{array}%
\right. \ \ \ \ \ \ \ \left\{ 
\begin{array}{c}
s_{4}\approx 0, \\ 
u_{4}\approx 0,%
\end{array}%
\right. \\
&&\left\{ 
\begin{array}{c}
s_{5}\approx -0.55944, \\ 
u_{5}\approx 1.021120,%
\end{array}%
\right. \ \ \ \ \ \left\{ 
\begin{array}{c}
s_{6}\approx -1.00000, \\ 
u_{6}\approx 0.774050,%
\end{array}%
\right. \ \ \ \ \ \left\{ 
\begin{array}{c}
s_{7}\approx -0.90172, \\ 
u_{7}\approx -0.02948,%
\end{array}%
\right. \ \ \ \ \ \left\{ 
\begin{array}{c}
s_{8}\approx -2.15794, \\ 
u_{8}\approx 1.743330,%
\end{array}%
\right. \\
&&\left\{ 
\begin{array}{c}
s_{9}\approx -1.00000, \\ 
u_{9}\approx -0.93984,%
\end{array}%
\right. \ \ \ \ \ \left\{ 
\begin{array}{c}
s_{10}\approx -0.73611, \\ 
u_{10}\approx -1.01418,%
\end{array}%
\right. \ \ \ \left\{ 
\begin{array}{c}
s_{11}\approx -1.08810, \\ 
u_{11}\approx -1.75309,%
\end{array}%
\right. \ \ \ \ \left\{ 
\begin{array}{c}
s_{12}\approx -1.00000, \\ 
u_{12}\approx -3.39671,%
\end{array}%
\right. \  \\
&&\left\{ 
\begin{array}{c}
s_{13}\approx 1.011050, \\ 
u_{13}\approx -0.78102,%
\end{array}%
\right. \ \ \ \ \left\{ 
\begin{array}{c}
s_{14}\approx 0.866840, \\ 
u_{14}\approx -1.26993.%
\end{array}%
\right. \ \ \ \ 
\end{eqnarray*}%
The shows that there is no point of maxima in $\left( 0,1\right) \times
\left( 0,1\right) .$ Hence $h$ has no critical point in $\left( 0,1\right)
\times \left( 0,1\right) .$\vspace{1.5mm}

\textbf{B.} Now we consider the case when 
\begin{equation*}
\left\vert \left( 1-t{_{{1}}^{2}}\right) \left( 31\left\vert t_{{2}%
}\right\vert ^{2}+225\right) -32t{_{{1}}^{2}}\overline{t_{{2}}}]\right\vert
-32\left\vert 8\left( 1-\left\vert t{_{{1}}}\right\vert ^{2}\right) t_{{2}}-t%
{_{{1}}^{2}}\right\vert <0.
\end{equation*}%
Then using the fact that $\left\vert t_{{3}}\right\vert \leq 1$ we obtain
from $\left( \ref{H3.3}\right) $ that 
\begin{equation*}
\left\vert H_{3}(1)\left( f^{-1}\right) \right\vert \leq \frac{1}{8640}%
h_{1}\left( \left\vert t{_{{1}}}\right\vert {,}\left\vert t{_{{2}}}%
\right\vert \right) ,
\end{equation*}%
where 
\begin{eqnarray*}
h_{1}(s,u) &=&76288s^{6}{+}64u\left( 1-s^{2}\right) [1740s^{4}{+6}%
s^{2}\left( 2986-1447s^{2}\right) u \\
&&+4\left( 1621s^{4}-2189s^{2}+1216\right) u^{2}+2511s^{2}\left(
1-s^{2}\right) u^{3}{]} \\
&&+10368\left( 1-s^{2}\right) \left( 1-u^{2}\right) [31\left( 1-s^{2}\right)
su^{2}+2\left( 3+13s^{2}\right) su+57s^{3}{]} \\
&&+165888\left( 1-s^{2}\right) \left( 1-u^{2}\right) \left[ 8\left(
1-s^{2}\right) u+s^{2}\right] \text{.}
\end{eqnarray*}%
We now show that $\max h_{1}\left( s,u\right) <1166400$ when $\left(
s,u\right) \in \left[ 0,1\right] \times \left[ 0,1\right] .$\vspace{1.2mm}

\textbf{I}. On the vertices of $\left[ 0,1\right] \times \left[ 0,1\right] $%
, we have 
\begin{equation*}
h_{1}\left( {0,}0\right) =0,\ \ \ h_{1}(0,1)=311296,\ \ \ h_{1}(1,0)=76288,\
\ \ h_{1}(1,1)=76288\text{.}
\end{equation*}%
\textbf{II.} On the sides of $\left[ 0,1\right] \times \left[ 0,1\right] $,
we have the following cases

(i) $h_{1}(0,u)=311296u^{3}+132710u(1-u^{2})\leq h_{1}(0,\frac{3}{62}\sqrt{%
186})=\frac{1327104}{31}\sqrt{186},$ $u\in \left( 0,1\right) $.

(ii) $h_{1}(1,u)=76288,$ $u\in \left( 0,1\right) $.

(iii) $h_{1}(s,0)=76288s^{6}+10368s^{2}(1-s^{2})(57s+16)\lesssim
h_{1}(.8000079381,0)\approx 167147.7331,~s\in \left( 0,1\right) $.

(iv) $h_{1}(s,1)=266304s^{6}-936960s^{4}+435648s^{2}+311296\leq
h_{1}(0,0.51157)\approx 365908.5766,~s\in \left( 0,1\right) $.

\textbf{III.} Now we consider the case for $\left( s,u\right) \in \left(
0,1\right) \times \left( 0,1\right) $. We find all solutions of system of
nonlinear equations 
\begin{eqnarray*}
\frac{\partial h_{1}}{\partial s} &=&321408(s^{2}-1)(3s^{3}-5s^{2}-s+1)u^{4}
\\
&&-768(3242s^{5}-1755s^{4}+1832s^{3}+810s^{2}-4642s+81)u^{3} \\
&&384(8682s^{5}+11880s^{4}-16004s^{3}-9639s^{2}+5108s+837)u^{2} \\
&&-768(870s^{5}+1755s^{4}-7492s^{3}-810s^{2}+6912s-81)u \\
&&+384s(1192s^{4}-7695s^{3}-1728s^{2}+4617s+864) \\
&=&0\text{,}
\end{eqnarray*}%
and 
\begin{eqnarray*}
\frac{\partial h_{1}}{\partial u} &=&642816(s+1)^{2}(s-1)^{2}s(s-2)u^{3} \\
&&+768(1-s^{2})(1621s^{4}-1053s^{3}+2995s^{2}-243s-3968)u^{2} \\
&&+768s(1-s^{2})(1447s^{3}+2376s^{2}-2554s-837)u \\
&&+768(1-s^{2}))(145s^{4}+351s^{3}-1728s^{2}+81s+1728) \\
&=&0\text{.}
\end{eqnarray*}%
We have the following solutions 
\begin{eqnarray*}
&&\left\{ 
\begin{array}{c}
s_{1}\approx 1.43410, \\ 
u_{1}\approx 1.05959,%
\end{array}%
\right. \ \ \ \ \ \ \ \left\{ 
\begin{array}{c}
s_{2}\approx 1.00000, \\ 
u_{2}\approx 1.69156,%
\end{array}%
\right. \ \ \ \ \ \ \ \left\{ 
\begin{array}{c}
s_{3}\approx 0.08804, \\ 
u_{3}\approx 0.66270,%
\end{array}%
\right. \ \ \ \ \ \ \left\{ 
\begin{array}{c}
s_{4}\approx 7.18284, \\ 
u_{4}\approx 0.10233,%
\end{array}%
\right. \\
&&\left\{ 
\begin{array}{c}
s_{5}\approx -0.57007, \\ 
u_{5}\approx 1.026200,%
\end{array}%
\right. \ \ \ \ \ \left\{ 
\begin{array}{c}
s_{6}\approx -1.00000, \\ 
u_{6}\approx 0.767320,%
\end{array}%
\right. \ \ \ \ \ \left\{ 
\begin{array}{c}
s_{7}\approx -2.29022, \\ 
u_{7}\approx 1.935860,%
\end{array}%
\right. \ \ \ \ \ \left\{ 
\begin{array}{c}
s_{8}\approx -1.14653, \\ 
u_{8}\approx 5.513580,%
\end{array}%
\right. \\
&&\left\{ 
\begin{array}{c}
s_{9}\approx -1.00000, \\ 
u_{9}\approx -0.89449,%
\end{array}%
\right. \ \ \ \ \ \left\{ 
\begin{array}{c}
s_{10}\approx -1.02196, \\ 
u_{10}\approx -1.25556,%
\end{array}%
\right. \ \ \ \ \left\{ 
\begin{array}{c}
s_{11}\approx -1.00000, \\ 
u_{11}\approx -1.91450,%
\end{array}%
\right. \ \ \ \left\{ 
\begin{array}{c}
s_{12}\approx -0.01501, \\ 
u_{12}\approx -0.66009,%
\end{array}%
\right. \\
&&\left\{ 
\begin{array}{c}
s_{13}\approx 0.891700, \\ 
u_{13}\approx -0.67707,%
\end{array}%
\right. \ \ \ \ \left\{ 
\begin{array}{c}
s_{14}\approx 1.002780, \\ 
u_{14}\approx -1.37766,%
\end{array}%
\right.
\end{eqnarray*}%
The only point in $\left( 0,1\right) \times \left( 0,1\right) $ is $\left(
s_{3},u_{3}\right) $ and we see that $h_{1}(s_{3},u_{3})=588255.0757.$ All
cases have been dealt with, $\left( \ref{H3.0}\right) $ holds. \medskip

To see that $\left( \ref{H3.0}\right) $ is sharp, consider $f_{\ast }:%
\mathbb{D}\rightarrow 
\mathbb{C}
$ given by 
\begin{equation*}
\left( zf_{\ast }^{\prime }\left( z\right) \right) ^{\prime }=\frac{1+z^{3}}{%
1-z^{3}}.
\end{equation*}%
Then $f_{\ast }\left( z\right) =z+\frac{1}{8}z^{4}+\frac{2}{49}z^{7}+\cdots $
with $a_{2}=a_{3}=0$ and $a_{4}=\dfrac{1}{8},\ a_{5}=0$ which from $\left( %
\ref{in3}\right) -\left( \ref{in4}\right) $ gives $A_{2}=A_{3}=0$ and $%
A_{4}=-\dfrac{1}{8},\ A_{5}=0.$ Now from $\left( \ref{h31}\right) ,$ we have 
$|H_{3}(1)\left( f_{\ast }^{-1}\right) |=\frac{1}{64}.$ This completes the
proof of the theorem.
\end{proof}

\end{document}